\documentclass[12pt]{article}
\usepackage{amssymb}
\usepackage{amsmath}
\usepackage{amsfonts}

\newtheorem{theorem}{Theorem}

\newtheorem{definition}[theorem]{Definition}

\newtheorem{lemma}[theorem]{Lemma}

\newtheorem{remark}[theorem]{Remark}

\newenvironment{proof}[1][Proof]{\noindent\textbf{#1.} }{\ \rule{0.5em}{0.5em}}

\begin{document}

\title{Pseudo-differential operators and related additive geometric
stable processes}
\author{Luisa Beghin \thanks{%
luisa.beghin@uniroma1.it, Department of Statistical Sciences, Sapienza -
University of Rome } \and Costantino Ricciuti \thanks{%
costantino.ricciuti@uniroma1.it, Department of Statistical Sciences,
Sapienza -University of Rome}}
\maketitle

\begin{abstract}
Additive processes are obtained from L\'{e}vy
ones by relaxing the condition of stationary increments, hence they are spatially (but not
temporally) homogeneous. By analogy with the case of time-homogeneous Markov processes,
one can define an infinitesimal generator, which is, of course,
 a time-dependent operator. Additive versions of stable and Gamma processes
have been considered in the literature. We
introduce here time-inhomogeneous generalizations of the well-known geometric
stable process, defined by means of time-dependent versions of fractional
pseudo-differential operators of logarithmic type. The local L\'{e}vy measures are
expressed in terms of Mittag-Leffler functions or $H$-functions with time-dependent
parameters. This article also presents some results
 about propagators related to additive processes.

\ \newline
\emph{AMS Subject Classification (2010):} 60G52; 60G51; 26A33.\newline
\emph{Keywords:} Time-inhomogeneous processes; geometric stable
distributions; fractional logarithmic operator; additive processes;
variance gamma process.
\end{abstract}



\section{Introduction}

Geometric stable random variables (GS r.v.'s) have been studied
since the Eighties and widely applied, in particular, in modelling
data with heavy-tails behavior, in mathematical finance and other
fields of research (see \cite{KOZ3} and \cite{KOZ}, for the
univariate and multivariate cases, respectively, and also
\cite{Meer1}). Indeed, the GS laws are characterized by heavy
tails, unboundedness at zero and by stability properties (with
respect to geometric
summation). \\
The GS process can be defined by means of the $\alpha$-stable process, as follows.
Let us consider an $\alpha $-stable process $S_{\alpha ,\theta }:=\left\{
S_{\alpha ,\theta }(t),t\geq 0\right\} $ which has (according to Feller's parametrization) the
following characteristic function
\begin{equation}
\mathbb{E}e^{i\xi S_{\alpha ,\theta }(t)}=\exp \{-t|\xi |^{\alpha
}e^{isign(\xi )\pi \theta /2}\}=:\exp \{-t\psi _{\alpha ,\theta }(\xi
)\}\qquad \alpha \in (0,2),\alpha \neq 1 .
\label{funzione caratteristica processo stabile}
\end{equation}%
where $|\theta|\leq min\{ \alpha, 2-\alpha \}$.
For $\alpha =1$, the characteristic function can be written in the form (%
\ref{funzione caratteristica processo stabile}) in the symmetric case only, i.e. for $%
\theta =0$. So, in order to have unified formulas, we exclude the case $%
\alpha =1$ from our exposition.  We will denote by $G_{\alpha ,\theta }:=\left\{ G_{\alpha ,\theta }(t),t\geq
0\right\} $ the univariate GS process; which can be represented (see \cite%
{GRZ}) as%
\begin{equation}
G_{\alpha ,\theta }(t):=S_{\alpha ,\theta }(\mathit{\Gamma }(t)),\qquad
t\geq 0,  \label{sub}
\end{equation}%
where \textit{$\Gamma $}$:=\left\{ \mathit{\Gamma }(t),t\geq 0\right\} $ is
an independent Gamma subordinator, with density%
\begin{equation}
f_{\mathit{\Gamma }(t)}(x)=\frac{e^{-x/b}x^{at-1}}{\Gamma (at)b^{at}},\qquad
x,t\geq 0,\text{ }a,b>0  \label{rr}
\end{equation}%
and characteristic function%
\begin{equation*}
\mathbb{E}e^{i\xi \mathit{\Gamma }(t)}=\left( 1-i\xi b\right) ^{-at}.
\end{equation*}%
We recall that the L\'{e}vy measure of $\mathit{\Gamma }$ is given by $\nu _{%
\mathit{\Gamma }}(dx)=ax^{-1}e^{-x/b}dx.$ We will put for simplicity $a=1.$
As a consequence of (\ref{sub}), $G_{\alpha ,\theta }$ is a L\'{e}vy process (see, for
example, \cite{SIK}), with characteristic exponent
\begin{equation}
\eta _{G_{\alpha ,\theta }}(\xi ):=\frac{1}{t}\ln \mathbb{E}e^{i\xi
G_{\alpha ,\theta }(t)}=-\ln \left( 1+b\psi _{\alpha ,\theta }(\xi )\right)
,\qquad \xi \in \mathbb{R}.  \label{ss}
\end{equation}%
We note that, for $\theta =0,$ the process $G_{\alpha ,\theta }$
is symmetric; in particular, for $\alpha =2$, it reduces to the
well-known variance gamma (VG) process, which, by (\ref{sub}), can
be represented as $G_{2 ,0 }(t):=B(\mathit{\Gamma }(t))$, $t\geq
0$, where $B:=\left\{ B(t),t\geq0\right\}$ is a standard Brownian
motion, independent of $\Gamma$. The VG is applied in option
pricing, since it allows for a wider modelling of skewness and
kurtosis than the Brownian motion does. Moreover the variance
gamma process has been successfully applied in the modelling of
credit risk in structural models. The pure jump nature of the
process and the possibility to control skewness and kurtosis of
the distribution allow the model to price correctly the risk of
default of securities having a short maturity, something that is
generally not possible with structural models in which the
underlying assets follow a Brownian motion (for more details, see,
for example, \cite{PAP} and \cite{COT}). For $\alpha \neq 2$ the
symmetric GS law is also called Linnik distribution (see
\cite{KOZ4}).\\
On the other hand, in the completely positively asymmetric case, i.e. for $\theta= -\alpha$ and for $%
\alpha \in (0,1),$ the GS law reduces to the so-called
Mittag-Leffler distribution (see \cite{KOZ4}),
 while the corresponding process is called GS
subordinator (since it is increasing). Its L\'{e}vy density,
being of order $\alpha /x$ for $x$ near the origin, is almost integrable
near zero, so that the subordinator is very slow. Thus it can be used for time-changing
another process, in order to slow it down (see \cite{SIK}, \cite{KOZ5}).
New families of subordinators, with explicit transition semigroup, have been
defined and analyzed in \cite{BUR}, along with some useful examples.\\

We present here time-inhomogeneous versions of the GS process.
There is a wide literature inspiring this subject.
 In  \cite{LEG} and \cite{FAL} the authors studied the so called multistable processes,
 namely inhomogeneous extensions of  stable L\'{e}vy ones, which are obtained by letting
 the stability parameter vary in time (for a different kind of multistable processes consult \cite{falco2}).
Such processes turned out to be very useful in financial and physical applications,
where the data display jumps with varying intensity (for financial applications of
additive processes see, for example \cite{KOV}).  Moreover,  time inhomogeneous
versions of Gamma subordinators and VG processes can be respectively found in  \cite{CIN} and \cite{MAD2}.
On the same line of research, in \cite{ORT} the authors defined the so called
inhomogeneous subordinators, i.e. non decreasing additive processes, which  are used as
models of random time change to extend the theory of Bochner subordination.
A remarkable case is that of the multistable subordinator considered in \cite{MOL}
(for an application of  Markov processes time-changed by multistable subordinators see \cite{BEG4}).
For the sake of completeness, we mention a further approach to inhomogeneous L\'{e}vy
processes, which are defined by letting a parameter depend on time, through a
subordinator (see, for example, \cite{LEO}, in the Poisson case).

We will define two, alternative, time-inhomogeneous versions of the GS process. The first one
is obtained by letting the parameters $\alpha$ and $\theta$ vary in time; then the generator of
the process is defined by means of a Riesz-Feller space-fractional derivative
with time-varying parameters (i.e. $\alpha (t)$ and $%
\theta (t)$). We will obtain some results on the tails' behavior
of the density
and the L\'{e}vy measure (at least in the subordinator case, i.e. for $\alpha (t) \in (0,1)$ and $%
\theta (t)=-\alpha (t)$, $\forall t$) for the new process.
 However, as we will see, we cannot
recover, as special case, an inhomogeneous version of the VG process, since
only the standard VG process can be obtained by putting $\alpha (t)=2$ and $%
\theta (t)=0$, for any $t.$ Moreover, no subordinating relationship, similar to (\ref{sub}) can be
established. Therefore we consider a second inhomogeneous version of the GS process,
by letting the variance parameter $b$ depend on $t$ and keeping $\alpha$ and $\theta$ constant in time.
The two previous drawbacks are then overcome, since this new process can be constructed as a stable process
time-changed by an inhomogeneous Gamma subordinator. Moreover the L\'{e}vy measure can be evaluated,
in this case, for any value of $\theta$ and it will expressed in terms of H-functions.

The plan of the paper is the following. In section 2 we recall some facts on 
time-inhomogeneous Markov processes and related propagators, following the line of 
\cite{JAC1}, \cite{JAC2}, \cite{BOTT1}, \cite{BOTT} and \cite{BOTT3}, 
and we underline new aspects on the special case of additive processes. In section 3 we write the generators of the GS processes
 as fractional operators of logarithmic type, which have been treated, in \cite{BEG},
 \cite{BEG2} and \cite{BEG3}, in the homogeneous case.  In Section 4 and 5 we construct
 the two additive geometric stable processes and present all the related results,
  together with some relevant particular cases. \\

\section{Notation and preliminary results}

Additive processes are obtained from L\'{e}vy ones by relaxing the condition
of stationarity of the increments (see, for example, \cite{SAT}, p.47).
Indeed a process $X:=\{X(t),t\geq 0\}$ is said to be \textit{additive} if

\begin{itemize}
\item $X(0)=0$ almost surely (a.s.)

\item $X$ has independent increments

\item $X$ is stochastically continuous
\end{itemize}

Thus $X$ is a spatially (but not temporally) homogeneous Markov process.
In this paper, we deal with one-dimensional processes. For
any $0\leq s\leq t$, the distribution $\mu _{s,t}$ of the increment $%
X(t)-X(s)$ is such that
\begin{equation}
\mathbb{E}e^{ip(X(t)-X(s))}=\int_{\mathbb{R}} e^{ipx}\mu _{s,t}(dx)=e^{\int_{s}^{t}\eta
(p,\tau )d\tau }\qquad 0\leq s\leq t,\text{ }p\in \mathbb{R},  \label{ch4}
\end{equation}%
where
\begin{equation}
\eta (p,t)=ib_{t}p-\frac{1}{2}z_{t}p^{2}+\int_{\mathbb{R}}\bigl (%
e^{ipy}-1-ipy1_{[-1,1]}(y)\bigr)\nu _{t}(dy). \label{esponente caratteristico}
\end{equation}%
Here $(b_{t},z_{t},\nu _{t})$ is called the characteristic triplet of $X$.
In particular, $b_t\in \mathbb{R}$, $z_t >0$ and $\nu _{t}$ is the time-dependent L\'{e}vy measure, such that
\begin{equation*}
\int_{\mathbb{R}} (y^2\wedge 1)\nu _{t}(dy)<\infty \qquad \forall t\geq 0.
\end{equation*}%
An additive process is completely determined by the set of measures $\mu
_{s,t},0\leq s\leq t$, since, from (\ref{ch4}), all the finite-dimensional
distributions are completely specified. Indeed, let $%
0=t_{0}<t_{1}<t_{2}<...<t_{n}$ and $\xi _{j}\in \mathbb{R},$ for $j=1,...,n,$
then the $n$-times characteristic function can be written as
\begin{equation}
\mathbb{E}e^{i\bigl (\xi _{1}X(t_{1})+\xi _{2}X(t_{2})+...+\xi _{n}X(t_{n})\bigr )}=%
\mathbb{E}\, exp \biggl \{i\sum_{k=1}^{n}(X(t_{k})-X(t_{k-1}))\sum_{j=k}^{n}\xi _{j} \biggr \},
\label{ch}
\end{equation}%
where the right-hand side is obtained by simple algebraic manipulations.
Then, by independence of the increments, formula (\ref{ch}) reduces to
\begin{equation}
\prod_{k=1}^{n}e^{\int_{t_{k-1}}^{t_{k}}\eta \left(
\sum_{j=k}^{n}\xi _{j},\tau \right) d\tau }=exp \biggl \{\int_{\mathbb{R}}\eta %
\biggl(\sum_{j=1}^{n}\xi _{j}1_{[0,t_{j}]}(\tau ),\tau \biggr)d\tau \biggr \}.
\label{chh}
\end{equation}%
An interesting example of additive process is the so called multistable
process (here denoted as $\mathcal{S}^I_{\alpha , \theta }$),
recently studied in \cite{FAL} and \cite{LEG}.
It can be defined by letting the parameters $\alpha$
and $\theta$ in (\ref{funzione caratteristica processo stabile})
 be time-dependent. Indeed, by assigning the two functions
$t \to \alpha (t)$ and $t\to \theta (t)$ such that
$\alpha (t) \in (0,1)$ and $|\theta (t)| \leq min \bigl (\alpha (t), 2-\alpha (t)\bigr )$, for any $t\geq 0$,
the characteristic function reads
\begin{align*}
\mathbb{E}e^{i\xi \mathcal{S}^I_{\alpha , \theta }(t)}= e^{-\int _0^t \psi _{\alpha (\tau) \theta (\tau)}(\xi)d\tau}
\end{align*}
Hence $\mathcal{S}^I_{\alpha , \theta }$ is a time-inhomogeneous extension of stable processes,
having independent and non-stationary increments, which turned out to be very useful in applications.
In the symmetric case, where $\theta (t)=0 \, \forall t$ and $\psi_{\alpha (t), \theta(t)}=|\xi|^{\alpha (t)}$,
 the joint distribution has the following characterization (see \cite{LEG}, formula (6))
\begin{align*}
\mathbb{E}e^{i\bigl (\xi _{1}\mathcal{S}^I_{\alpha , \theta }(t_1)+\xi _{2}\mathcal{S}^I_{\alpha , \theta }
(t_2)+...+\xi _{n}\mathcal{S}^I_{\alpha , \theta }(t_n) \bigr )}=exp \biggl
\{-\int_{\mathbb{R}}|\sum_{j=1}^{n}\xi _{j}1_{[0,t_{j}]}(\tau )| ^{\alpha (\tau)}d\tau \biggr \},
\end{align*}
for $0<t_{1}<...<t_{n}$.

\subsection{Propagators and time-dependent generators}
 The link between Markov processes and pseudo-differential operators has been studied in \cite{JAC1} and \cite{JAC2}.
  The theory has been developed in subsequent papers, such as \cite{BOTT1}, 
  \cite{BOTT} and \cite{BOTT3}, which mostly regard the case of time-inhomogeneous Markov processes. The role of this section is to deepen  a particular case of time-inhomogeneous Markov processes, namely the additive ones.

Let $X$ be a univariate additive process and let
\begin{equation*}
p_{s,t}(dy)=P(X(t)\in d(x+y)|X(s)=x)
\end{equation*}%
be its transition probability, which is independent of $x$ as $X$ is
space-homogeneous. Observe that, since $X$ is additive,  $p_{s,t}(dy)$  coincides
(see \cite{SAT}, page 55, Thm 10.4) with the law of the increment $\mu _{s,t}(dy)=P(X(t)-X(s)\in dy)$
defined in the previous Section.

Clearly $X$ defines a propagator (namely, a two-parameters semigroup)
\begin{equation}
T_{s,t}f(x)=\int_{\mathbb{R}}f(x+y)p_{s,t}(dy)\qquad 0\leq s\leq t
\label{def}
\end{equation}%
for $f$ in the Banach space  $ C(\mathbb{R})$, equipped with the sup-norm. Of course, $T_{t,t}$ is the identity operator and,
for any $0\leq r\leq s\leq t$, the chain rule $T_{s,t}T_{r,s}=T_{r,t}$ holds.
The time-dependent generator of $T_{s,t}$ is defined (see \cite{KOLO}, page
48, and \cite{RUS}) as the operator
\begin{equation*}
\mathcal{A}_{t}f=\lim_{h\to 0^+}\frac{T_{t,t+h}f-f}{h}
\end{equation*}%
on a subset of $C(\mathbb{R})$ where the limit (meant in the sup-norm of $C(%
\mathbb{R}))$ exists.

Let $\mathcal{S}(\mathbb{R})$ be the Schwartz space of infinitely
differentiable functions on $\mathbb{R}$, decreasing at infinity (together
with all their derivatives) faster than any power. In analogy to the
standard theory of L\'{e}vy processes,  by restricting the
domain of $T_{s,t}$ to $\mathcal{S}(\mathbb{R})\subset C(\mathbb{R})$, one
can easily find the form of the time-dependent generator of $T_{s,t}$, by means of the
representation as pseudo-differential operator. Indeed, the crucial point is that the Fourier transform
\begin{equation*}
\tilde{f}(p)=\frac{1}{\sqrt{2\pi }}\int_{\mathbb{R}}e^{-ipx}f(x)dx
\end{equation*}%
is a bijective operation on $\mathcal{S}(\mathbb{R})$ and the inverse
transform is defined as
\begin{equation*}
f(x)=\frac{1}{\sqrt{2\pi }}\int_{\mathbb{R}}e^{ipx}\tilde{f}(p)dp.
\end{equation*}

In the following lemma we extend well-known results on L\'{e}vy processes to
the additive case, by writing the pseudo-differential form for propagators and generators.

\begin{lemma}
\label{lemma operatori pseudodifferenziali} Let $X$ be an additive process with
characteristic exponent  $\eta (p,t)$ (given in (\ref{esponente caratteristico}))
continuous in $t$ and such that    $|\eta (p,t)|$ is bounded  in $t$. Let $T_{s,t}$
be the associated propagator. Then, for any $f\in \mathcal{%
S}(\mathbb{R})$,

i) the propagator has the following representation
\begin{align}  \label{rappresentazione pseudodifferenziale del propagatore}
T_{s,t}f(x)= \frac{1}{\sqrt{2\pi}} \int_{\mathbb{R}} e^{ipx} e^{\int _s^t
\eta (p,\tau)d\tau }\tilde{f}(p)dp
\end{align}

ii) the time-dependent generator has the form
\begin{equation}
\mathcal{A}_{t}f(x)=\frac{1}{\sqrt{2\pi }}\int_{\mathbb{R}}e^{ipx}\eta (p,t)%
\tilde{f}(p)dp \label{rappresentazione pseudodifferenziale del generatore}
\end{equation}

iii) the generator can be equivalently written as
\begin{equation*}
\mathcal{A}_{t}f(x)=\frac{1}{2}z_{t}\frac{\partial ^{2}}{\partial x^{2}}%
f(x)+b _{t}\frac{\partial f(x)}{\partial x}+\int_{\mathbb{R}}\left[
f(x+y)-f(x)-y\frac{\partial }{\partial x}f(x)1_{\left\{ |y|\leq 1\right\} }%
\right] \nu _{t}(dy),
\end{equation*}
\end{lemma}

\begin{proof}
i) The Fourier transform of $T_{s,t}f(x)$ is
\begin{align*}
\frac{1}{\sqrt{2\pi }}\int_{\mathbb{R}}dx\, e^{-ipx}\int_{\mathbb{R}%
}f(x+y)p_{s,t}(dy)& =\frac{1}{\sqrt{2\pi }}\int_{\mathbb{R}%
}p_{s,t}(dy)e^{ipy}\int_{\mathbb{R}}dx\,f(x+y)e^{-ip(x+y)} \\
& =exp\biggl (\int_{s}^{t}\eta (p,\tau )d\tau \biggr)\tilde{f}(p)
\end{align*}%
which proves (i), by Fourier inversion.

ii) Let $f$ be in the domain of the generator, i.e. suppose that $\frac{%
T_{t,t+h}f-f}{h}$ converges uniformly to $\mathcal{A}_{t}f$ as $h\to 0^+$. Then it is
sufficient to use the representation (\ref{rappresentazione
pseudodifferenziale del propagatore}) and to apply the pointwise limit
\begin{equation*}
\mathcal{A}_{t}f(x)=\lim_{h \to 0^+}\frac{T_{t,t+h}f(x)-f(x)}{h}%
=\lim_{h\to 0^+}\frac{1}{\sqrt{2\pi }}\int_{\mathbb{R}}e^{ipx}\,\frac{%
e^{\int_{t}^{t+h}\eta (p,\tau )d\tau }-1}{h}\tilde{f}(p)dp
\end{equation*}%
The result immediately follows by exchanging the limit and the integral,
which is permitted by the dominated convergence theorem. Indeed, by the mean value theorem,
\begin{align*}
\biggl |e^{ipx}\,\frac{e^{\int_{t}^{t+h}\eta (p,\tau )d\tau }-1}{h}\tilde{f}%
(p)\biggr |& = \biggl |e^{ipx}\,\frac{e^{\eta (p,\tau ^* )h }-1}{h}\tilde{f}%
(p)\biggr |      \leq \biggl |\eta (p,\tau ^{\ast })\tilde{f}(p)\biggr |  \\ & \leq  \biggl |\eta (p,\tau _{max })\tilde{f}(p)\biggr |    \leq \biggl
|C(1+|p|^{2})\tilde{f}(p)\biggr |
\end{align*}%
where $\tau _{max}$ is the point where $|\eta (p,\tau)|$ has its maximum in
$\tau$ and, in the last inequality, we used [\cite{APP}, p.31]. Moreover
\begin{equation*}
\int_{\mathbb{R}}|C(1+|p|^{2})\tilde{f}(p)dp<\infty
\end{equation*}%
since $|C(1+|p|^{2})\tilde{f}(p)|$ is clearly a Schwartz function.

iii) It is sufficient to insert (\ref{esponente caratteristico}) into (\ref
{rappresentazione pseudodifferenziale del generatore}) and the result is
obtained by Fourier inversion.
\end{proof}

\vspace{0.5cm}

Some of the processes considered here are symmetric. Recall that an additive
process is symmetric if the transition probability is such that $p_{s,t}(dy)=p_{s,t}(-dy)$, for any $s,t\in
\mathbb{R}^{+}$. In this case the characteristic
function of its increments (\ref{ch4}) reads%
\begin{equation}
\mathbb{E}e^{i\xi (X(t)-X(s))}=e^{\int_{s}^{t}\eta (|\xi |,\tau )d\tau
}\qquad 0\leq s\leq t.
\end{equation}%
The following result extends to symmetric additive processes a well-known
property enjoyed by symmetric L\'{e}vy ones (see \cite{APP} p.178).

\begin{lemma} \label{operatori simmetrici}
Let $X$ be a symmetric additive process. Then the associated propagator $%
T_{s,t}$ is self-adjoint in $\mathbb{L}_{2}$.
\end{lemma}

\begin{proof}
By (\ref{def}) we can write
\begin{equation*}
T_{s,t}f(x)=\int_{\mathbb{R}}f(x+y)p_{s,t}(dy)=   \int_{\mathbb{R}}f(x+y)p_{s,t}(-dy)= \int_{\mathbb{R}}f(x-y)p_{s,t}(dy)  .
\end{equation*}%
It is enough to show that $<T_{s,t}f;g>=<f;T_{s,t}g>$, for any $f,g\in
\mathbb{L}_{2}$ (where $<\cdot ;\cdot >$ denotes the scalar product in $%
\mathbb{L}_{2}$). Then
\begin{align*}
&<T_{s,t}f;g> =\int_{\mathbb{R}}T_{s,t}f(x)g(x)dx=\int_{\mathbb{R}%
}\int_{\mathbb{R}}f(x-y)p_{s,t}(dy)g(x)dx \\
& =\int_{\mathbb{R}%
}\int_{\mathbb{R}}f(x')p_{s,t}(dy)g(x'+y)dx'= \int _{\mathbb{R}} f(x')T_{s,t}g(x')dx'    =<f;T_{s,t}g>
\end{align*}%
where the order of integration has been inverted by the Fubini theorem.
\end{proof}

\subsection{Time-change by inhomogeneous subordinators}

The so-called non-homogeneous subordinators (i.e. non-decreasing additive
processes, which can be used as random-time) have been studied in
\cite{ORT}. Their  transition measure $\mu_{s,t}$  has Laplace transform
\begin{align*}
\int _{0}^\infty e^{-\eta z} \mu_{s,t}(dz)=   e^{-\int _s^t f(\eta, \tau) d\tau},\qquad \eta >0,
\end{align*}
where
$f(\eta,\tau)= \int _0^\infty (1-e^{-s\eta}) \nu _{\tau}(ds)$
is a Bernstein function in the variable $\eta$. A remarkable example is the
 multistable subordinator studied in \cite{MOL}, which corresponds to
 $f(\eta, \tau)= \eta ^{\alpha (\tau)}$, where $\alpha (\tau) \in (0,1)$.

In \cite{ORT} the authors considered propagators defined by the Bochner integral
\begin{equation}
\mathcal{T}_{s,t}f=\int_{0}^{\infty }T_{z}f\,\,\mu _{s,t}(dz),  \label{hhh}
\end{equation}%
where $T_{t}$ is a contraction semigroup (not necessarily associated to a
stochastic process) acting on a generic Banach space, and $\mu _{s,t}$ is
the increment law of a non-homogeneous subordinator. The operator (\ref{hhh}%
) is a subordinated propagator (not necessarily associated to a stochastic
process) acting on a generic Banach space. In \cite{ORT}, Theorem 4.1, the
form of the time-dependent generator of (\ref{hhh}) is found, by considering
an Hilbert space as domain of $\mathcal{T}_{s,t}$ and assuming $\mathcal{T}_{s,t}$
to be self-adjoint. This result generalizes
the Phillips theorem (see \cite{SAT}) holding for one-parameter subordinated
semigroups.

In view of what
follows, it is useful to prove a more general result which is valid for not
necessarily self-adjoint propagators.  The price we pay to make this
generalization is to narrow the attention to propagators
acting on
$\mathcal{S}(\mathbb{R})$ only. Moreover we only consider propagators
associated to time-changed  L\'{e}vy processes.

\begin{lemma}
Let $M$ be a L\'{e}vy process associated to the semigroup $T_{t}$ on $\mathcal{S}(\mathbb{R})$,
having characteristic function $\mathbb{E}%
e^{i\xi M(t)}=e^{t\eta (\xi)}$, $\xi \in \mathbb{R}$ and let $H$ be an inhomogeneous subordinator with
Laplace transform $\mathbb{E}e^{-p(H(t)-H(s))}=e^{-\int_{s}^{t}f(p,\tau
)d\tau }$, $p\in \mathbb{R}^+$. Then the additive process $M(H(t))$ has time-dependent generator
\begin{equation*}
\mathcal{L}_{t}h(x)=\int_{0}^{\infty
}(T_{z}h(x)-h(x))\nu _{t}(dz),\qquad h\in \mathcal{S}(\mathbb{R}).
\end{equation*}
\end{lemma}

\begin{proof}
By a standard conditioning argument, we have
\begin{equation*}
\mathbb{E}e^{ipM(H(t))}=e^{-\int_{0}^{t}f(-\eta (p),\tau )d\tau }.
\end{equation*}%
It is now sufficient to apply Lemma \ref{lemma operatori pseudodifferenziali} to the
characteristic exponent $-f(-\eta (p),\tau )$. Indeed, by using expression
(\ref{rappresentazione pseudodifferenziale del generatore}), we have
\begin{align*}
\mathcal{L}_t h(x)= \frac{1}{\sqrt{2\pi}}\int _{\mathbb{R}}e^{ipx}
\biggl ( -f(-\eta (p),t) \biggr ) \tilde{h}(p)dp= \frac{1}{\sqrt{2\pi}}\int _{\mathbb{R}}e^{ipx}
 \biggl ( \int _0^\infty (e^{\eta (p)s}-1)\nu_t(ds) \biggr ) \tilde{h}(p)dp
\end{align*}
and recalling that
 \begin{align*}
 T_sh(x)= \frac{1}{\sqrt{2\pi}} \int _{\mathbb{R}}e^{ipx}e^{\eta (p)s}\tilde{h}(p)dp
 \end{align*}
  we obtain the result.
\end{proof}

\section{Fractional derivatives with time-varying parameters}

\subsection{Multistable processes and their generators}

The \emph{Riesz-Feller (RF) fractional
derivative}  is a pseudo-differential operator defined in \cite{MAI} by means of its
Fourier transform. For any $f\in S(\mathbb{R})$, the RF
fractional derivative $\mathcal{D}^{\alpha ,\theta }$ (with $\alpha \in
(0,2],$ and $|\theta |\leq \min \{\alpha ,2-\alpha \}$), is defined by
 \footnote{To avoid confusion with formulas reported in \cite{MAI}, we point out that
 our definition of Fourier transform of a function $h$ is $\hat{h}(\xi)= 1/\sqrt{2\pi }
 \int _{\mathbb{R}} e^{-i\xi x}h(x)dx$  ,   while the authors in \cite{MAI} define the
 Fourier transform as $\hat{h}_M(\xi)= \int _{\mathbb{R}}e^{i\xi x}h(x)dx$, hence we
 have $\hat{h}(\xi)=\hat{h}_M(-\xi)/\sqrt{2\pi}$.  Then, for
 $h(x)= \mathcal{D}^{\alpha, \theta}_xf(x)$, the authors in \cite{MAI}
 write $\hat{h}_M(\xi)=-|\xi|^\alpha e^{i sgn(\xi) \pi\theta /2}\hat{f}_M(\xi)$,
 which, according to our definition of Fourier transform, becomes
\begin{align*}
 \hat{h}(\xi)=- \frac{1}{\sqrt{2 \pi}}|\xi|^\alpha e^{-i sgn(\xi) \pi\theta /2}\hat{f}_M(-\xi)=- |\xi|^\alpha e^{-i sgn(\xi) \pi\theta /2}\hat{f}(\xi)
\end{align*}
  }

\begin{align}
\frac{1}{\sqrt{2\pi }}\int_{\mathbb{R}}e^{-i\xi x}\mathcal{D}_{x}^{\alpha
,\theta }f(x)dx=-|\xi |^{\alpha }e^{-i\,sign(\xi )\theta \pi /2}\tilde{f}(\xi
),  \label{uno}
\end{align}%
namely its "symbol" reads
\begin{equation}
\widehat{\mathcal{D}^{\alpha ,\theta }}(\xi )=-|\xi |^{\alpha }e^{-i\,sign(\xi )\theta \pi /2}.  \label{due}
\end{equation}%

The adjoint of $\mathcal{D}^{\alpha, \theta}$ is defined as the operator $\overline{\mathcal{D}^{\alpha, \theta}}$ such that
\begin{align}
\int _\mathbb{R} \mathcal{D}^{\alpha, \theta}_x f(x) \, g(x)dx= \int _{\mathbb{R}}f(x)
\overline{\mathcal{D}_x^{\alpha, \theta}} g(x)dx \qquad  \forall f,g\in \mathcal{S}
(\mathbb{R}) \label{formula di aggiunzione}
\end{align}
and it is easy
 to check that  $\overline{\mathcal{D}^{\alpha, \theta}}$ is such that
\begin{align}
\frac{1}{\sqrt{2\pi }}\int_{\mathbb{R}}e^{-i\xi x} \overline{\mathcal{D}_{x}^{\alpha
,\theta }}f(x)dx=-\psi _{\alpha, \theta}(\xi)\tilde{f}(\xi)=    -|\xi |^{\alpha }e^{i\,sign(\xi )\theta \pi /2}\tilde{f}(\xi
). \label{aggiunta della derivataa}
\end{align}
Indeed, by writing $\mathcal{D}^{\alpha, \theta}_x f(x$) and $\overline{\mathcal{D}_x^{\alpha, \theta}}
g(x)$ as the inverse Fourier transforms of (\ref{uno}) and (\ref{aggiunta della derivataa}),
we can verify that (\ref{formula di aggiunzione}) is satisfied.

By inserting (\ref{aggiunta della derivataa}) into (\ref{rappresentazione pseudodifferenziale del generatore})
(restricted to the time-homogeneous case), we see that the operator
 $\overline{\mathcal{D}^{\alpha, \theta}}$ is the generator of the
stable process $\mathcal{S}_{\alpha ,\theta }$. Therefore, if $T_t$
is the stable semigroup, we have that $q(x,t)=T_tf(x)= \mathbb{E}
(f(S_{\alpha, \theta}(t))|S_{\alpha, \theta}(0)=x)$ solves
\begin{align*}
\frac{\partial}{\partial t} q(x,t)= \overline{\mathcal{D}_x^{\alpha, \theta}} q(x,t) \qquad q(x,0)=f(x),
\end{align*}
while  the stable density $p_{\alpha, \theta}(x,y,t)$ (which obviously depends on
 $y-x$ only, since $S_{\alpha, \theta}$ is a L\'evy process) solves the backward equation
\begin{align*}
\frac{\partial}{\partial t} p_{\alpha, \theta}(x,y,t)=
\overline{\mathcal{D}_x^{\alpha, \theta}}p_{\alpha, \theta}(x,y,t) \qquad p_{\alpha,\theta}(x,y,0)= \delta (x-y).
\end{align*}
The picture is completed by the forward equation  (where the operator
 on the right-hand side acts on the forward variable $y$)
\begin{align*}
\frac{\partial}{\partial t} p_{\alpha, \theta}(x,y,t)=
\mathcal{D}_y^{\alpha, \theta}p_{\alpha, \theta}(x,y,t) \qquad p_{\alpha, \theta}(x,y,0)= \delta (x-y)
\end{align*}
which has been studied in  \cite{MAI} (although with a different notation and in a different setting).

 The previous facts can be extended to the case of time-varying fractional index $\alpha (t)$ and
 asymmetry parameter $\theta
(t)$, both assumed to be continuous. To this aim we give the following definition.

\begin{definition}
\label{def_1} (RF fractional derivative with varying parameters) Let $\,t\to \alpha (t)$ and
 $\,t\to \theta (t)$ be two continuous functions such that $\alpha
(t)\in (0,2]$ and $|\theta (t)|\leq \min \{\alpha (t),2-\alpha (t)\},$ $%
\forall t\geq 0$. We define $\mathcal{D}^{\alpha (t),\theta (t)} $ by
its Fourier transform%
\begin{equation}
\frac{1}{\sqrt{2\pi}}\int _{\mathbb{R}}e^{-i\xi x}\, \mathcal{D}_{x}^{\alpha
(t),\theta (t)} f(x)dx =-|\xi |^{\alpha (t) }e^{-i\,sign(\xi )\theta (t) \pi
/2} \tilde{f}(\xi), \qquad f\in S(\mathbb{R}),  \label{sym2}
\end{equation}%
for any $t\geq 0,$ so that its symbol can be written as%
\begin{equation}
\widehat{\mathcal{D}^{\alpha (t),\theta (t)}}(\xi )=-|\xi |^{\alpha (t)}e^{-i\,sign(\xi )\theta (t)\pi /2}.
\label{sym}
\end{equation}
\end{definition}

Thanks to a suitable regularization of hyper-singular integrals,
the fractional derivative $\mathcal{D}^{\alpha (t),\theta (t)} $
 defined above can be also represented as follows (see [\cite{MAI}, (2.8)]):

\begin{eqnarray}
\mathcal{D}_{x}^{\alpha (t),\theta (t)}u(x) &=&\frac{\Gamma (1+\alpha (t))}{\pi }\sin \left( \frac{\left[ \alpha
(t)+\theta (t)\right] \pi }{2}\right) \int_{0}^{+\infty }\frac{f(x+z)-f(x)}{%
z^{1+\alpha (t)}}dz  \label{rie} \\
&&+\frac{\Gamma (1+\alpha (t))}{\pi }\sin \left( \frac{\left[ \alpha
(t)-\theta (t)\right] \pi }{2}\right) \int_{0}^{+\infty }\frac{f(x-z)-f(x)}{%
z^{1+\alpha (t)}}dz, \notag
\end{eqnarray}

The adjoint of the time-varying RF fractional derivative, say $\overline{\mathcal{D}^{\alpha (t), \theta (t)}}$, is defined by
\begin{align}
\frac{1}{\sqrt{2\pi }}\int_{\mathbb{R}}e^{-i\xi x} \overline{\mathcal{D}_{x}^{\alpha (t)
,\theta (t) }}f(x)dx=-\psi _{\alpha (t), \theta (t)}(\xi)\tilde{f}(\xi)=    -|\xi |^{\alpha (t) }e^{i\,sign(\xi )\theta (t) \pi /2}\tilde{f}(\xi
). \label{aggiunta della derivata}
\end{align}
 By (\ref{rappresentazione pseudodifferenziale del generatore}),
 $\overline{\mathcal{D}_{x}^{\alpha (t),\theta (t) }}$ is the time-dependent generator of the so-called
multistable process $\mathcal{S}_{\alpha ,\theta }^{I}$ studied in \cite{LEG} and \cite{FAL}, since
its characteristic function reads
\begin{equation}
\mathbb{E}e^{i\xi S_{\alpha ,\theta }^{I}(t)}=\exp \left\{ -\int_{0}^{t}\psi
_{\alpha (s),\theta (s)}(\xi )ds\right\}.  \label{td}
\end{equation}

In the special case where $\alpha (t)\leq 1$ and $\theta (t)=-\alpha (t)$
for each $t\geq 0$, the process $S_{\alpha ,\theta }^{I}$ coincides with the
multistable subordinator given in \cite{MOL} and the generator reduces to the%
\textit{\ variable order (left-sided) Riemann-Liouville derivative},
\begin{equation}
\overline{\mathcal{D}_{x}^{\alpha (t),-\alpha (t)}}f(x):=\left\{
\begin{array}{l}
\frac{(-1)}{\Gamma (1-\alpha (t))}\frac{d}{dx}\int_{x}^{+\infty }\frac{f(z)}{%
(x-z)^{\alpha (t)}}dz,\quad \alpha (t)\in (0,1) \\
\frac{d}{dx}f(x),\quad \alpha (t)=1, \quad \forall t,%
\end{array}%
\right.  \label{crr}
\end{equation}%
with symbol $-(-i\xi )^{\alpha (t)}$.

Finally, in the symmetric case $\theta (t) =0 \,\, \forall t>0$,  the propagator
 is self-adjoint (in agreement to Lemma \ref{operatori simmetrici}), and its
 generator is given by the self-adjoint RF derivative (which is also known
 as Riesz derivative) $\mathcal{D}^{\alpha (t),0}= \overline{\mathcal{D}^{\alpha (t),0}}$,
 with symbol $-|\xi|^{\alpha (t)}$, having the following representation

\begin{align}
\mathcal{D}_{x}^{\alpha (t),0}f(x) =  \begin{cases}   \frac{\Gamma (1+\alpha (t))}{\pi }\sin \left( \frac{ \alpha
(t) \pi }{2}\right)  \int_{0}^{+\infty }\frac{f(x+z)+f(x-z)-2f(x)}{%
z^{1+\alpha (t)}}dz
 \qquad \alpha (t)\neq 2 \\ \frac{d^2}{dx^2}f(x) \qquad \alpha (t)=2 \label{derivative}
\end{cases}
\end{align}

\subsection{Fractional logarithmic operator with time-varying parameters}

By the Phillips' theorem, the classical Geometric Stable process $S_{\alpha
,\theta }(\Gamma )$ has generator
\begin{equation*}
\mathcal{G}_{b}^{\alpha ,\theta }u=\int_{0}^{\infty }(T_{s}u-u)\,\frac{%
e^{-s/b}}{s}ds
\end{equation*}%
where $T_{t}$ is the semigroup associated to $S_{\alpha ,\theta }$. Since
the characteristic function reads
\begin{equation*}
\mathbb{E}e^{i\xi S_{\alpha ,\theta }(\Gamma (t))}=exp\{-t\ln (1+b\psi
_{\alpha ,\theta }(\xi ))\},
\end{equation*}%
where $-\psi_{\alpha, \theta}(\xi)=-
|\xi|^{\alpha}e^{i\frac{\pi}{2} \theta sgn (\xi)}$ is the symbol
of $\overline{\mathcal{D}^{\alpha, \theta}}$, then, in the spirit
of operational functional calculus, we wonder if the generator can
be also written in the form of the fractional operator
\begin{equation}
\mathcal{G}_{b}^{\alpha ,\theta }=-\ln (1-b\overline{\mathcal{D}^{\alpha ,\theta }}).
\label{frac logaritmic operator}
\end{equation}%
But, in order to give sense to (\ref{frac logaritmic operator}), we need a
broader discussion, which regards a large class of subordinated semigroups.
Let $T_{t}$ be a contraction semigroup generated by $\mathcal{A}$ and let $%
T_{t}^{f}$ be the time changed semigroup, where $f(x)= \int _0^\infty (1-e^{-sx})\nu (ds)$ is the Bernstein
function of the underlying subordinator with L\'{e}vy measure $\nu $.
The Phillips' theorem states that the generator of $T_{t}^{f}$ is
\begin{equation*}
\mathcal{A}^{f}u=\int_{0}^{\infty }(T_{s}u-u)\nu (ds).
\end{equation*}%
where $Dom(\mathcal{A})\subset Dom(\mathcal{A}^{f})$. We now
wonder whether it is possible to write
$\mathcal{A}^{f}=-f(-\mathcal{A})$ in the sense of operational
functional calculus. According to \cite{Schilling 1}, the answer
is affirmative if $f$ is a complete Bernstein function.

We recall that a function $f$ on $(0,\infty)$ is said to be a complete Bernstein function if and only
if has the following analytic continuation on the upper complex half-plane
\begin{equation*}
f(z)=a+bz+\int_{0}^{\infty }\frac{z}{z+\tau }\,\sigma (d\tau ), \qquad Im\, z>0,
\end{equation*}%
for suitable constants $a,b\geq 0$ and a suitable measure $\sigma $ such that $\int_{0}^{\infty }\frac{\sigma
(d\tau )}{1+\tau }<\infty $ (for further details see \cite{Schilling 2},
ch.6).

By taking $a=0$, $b=0$ and $\sigma (d\tau )=\frac{d\tau }{\tau }1_{[b^{-1},\infty ]}$, we
have that $z\rightarrow \ln (1+bz)$ is a complete Bernstein function with
the following representation
\begin{equation*}
\ln (1+bz)=\int_{b^{-1}}^{\infty }\frac{z}{\tau (\tau +z)}d\tau.
\end{equation*}%
Hence, by using \cite{Schilling 1}, p.455, we obtain a nice integral representation
of $\mathcal{G}_b^{\alpha ,\theta }$ involving the adjoint of the RF derivative $\overline{\mathcal{D}
^{\alpha ,\theta }}$ and its resolvent $\mathcal{R}_{\tau }:=(\tau -\overline{\mathcal{D
}^{\alpha ,\theta }})^{-1}$ only:
\begin{equation}
\mathcal{G}_b^{\alpha ,\theta }=-\ln (1-b\overline{\mathcal{D}^{\alpha ,\theta
}})= \int_{b^{-1}}^{\infty }\frac{1}{\tau }\overline{\mathcal{D}^{\alpha ,\theta }}(\tau -
\overline{\mathcal{D}^{\alpha ,\theta }})^{-1}d\tau  \label{generatore del geometric stable}
\end{equation}%
which is valid on $Dom(\overline{\mathcal{D}^{\alpha ,\theta }})$. In the spirit of \cite%
{BEG}, \cite{BEG2} and \cite{BEG3}, we call (\ref{generatore del geometric stable}) \textit{%
fractional logarithmic operator}. Analogously, by considering a
generic semigroup $T_t$ generated by $\mathcal{A}$, subordination
to Gamma process produces the new generator
\begin{align}
-\ln (1-b\mathcal{A})=\int_{b^{-1}}^{\infty }\frac{1}{\tau }\mathcal{A} (\tau -%
\mathcal{A})^{-1}d\tau. \label{generatore logaritmico}
\end{align}

\begin{remark}

In order to strengthen the connection to fractional calculus, we observe
that for functions $f\in S(\mathbb{R})$ such that $\tilde{f}(\xi )$ is
compactly supported in $|\xi |<1/b^{\frac{1}{\alpha }}$, the symbol of the
generator can be expanded as
\begin{equation*}
\widehat{\mathcal{G}_{b}^{\alpha ,\theta }(\xi )}\tilde{f}(\xi )=-\ln (1+b\psi
_{\alpha ,\theta }(\xi ))\tilde{f}(\xi )=\sum_{n=1}^{\infty }\frac{%
(-1)^{n}b^{n}}{n}\psi _{\alpha ,\theta }^{n}(\xi )\tilde{f}(\xi ).
\end{equation*}%
Therefore $\mathcal{G}_{b}^{\alpha ,\theta }$ has the form of a powers' series
of fractional operators, i.e.
\begin{equation}
\mathcal{G}_{b}^{\alpha ,\theta }f(x)=\sum_{n=1}^{\infty }\frac{b^{n}}{n}%
\underbrace{\overline{\mathcal{D}_x^{\alpha ,\theta }}...\overline{\mathcal{D}_x^{\alpha
,\theta }}}_{n-times}f(x) \label{sviluppo in serie}
\end{equation}%
provided that the series converges uniformly. Regarding formula (\ref{sviluppo in
serie}), it must be taken into account that, for any $f\in \mathcal{S}(\mathbb{R})$, we
have
\begin{equation*}
\underbrace{\overline{\mathcal{D}^{\alpha ,\theta }}...\overline{\mathcal{D}^{\alpha
,\theta }}}_{j-times}f\in \mathcal{S}(\mathbb{R}),
\end{equation*}%
for any $j\in \mathbb{N}$. Formula (\ref{sviluppo in serie}) can be checked by
observing that
\begin{eqnarray*}
&&\frac{1}{\sqrt{2 \pi}}\int_{-\infty }^{+\infty }e^{-i\xi x}\sum_{n=1}^{\infty }\frac{b^{n}}{n}%
\underbrace{\overline{\mathcal{D}_x^{\alpha ,\theta }}...\overline{\mathcal{D}_x^{\alpha
,\theta }}}_{n-times}f(x)dx \\
&=&\frac{1}{\sqrt{2 \pi}}\sum_{n=1}^{\infty }\frac{b^{n}}{n}\int_{-\infty }^{+\infty }e^{-i\xi x}%
\underbrace{\overline{\mathcal{D}_x^{\alpha ,\theta}}...\overline{\mathcal{D}_x^{\alpha
,\theta }}}_{n-times}f(x)dx \\
&=&[\text{by (\ref{aggiunta della derivata})}] \\
&=&-\frac{1}{\sqrt{2 \pi}}\sum_{n=1}^{\infty }\frac{b^{n}}{n}\psi _{\alpha ,\theta }(\xi )\tilde{f}%
(\xi )\int_{-\infty }^{+\infty }e^{-i\xi x}\underbrace{\overline{\mathcal{D}_x%
^{\alpha ,\theta }}...\overline{\mathcal{D}_x^{\alpha ,\theta }}}%
_{(n-1)-times}f(x)dx \\
&=&\sum_{n=1}^{\infty }\frac{(-b)^{n}}{n}\psi _{\alpha ,\theta }(\xi )^{n}%
\tilde{f}(\xi )^{n}.
\end{eqnarray*}
\end{remark}

In the following sections, we study propagators generated by two possible
time-inhomogeneous extensions of (\ref{generatore del geometric stable}). The
first one is obtained by letting $\alpha $ and $\theta $ depend on $t$, as
follows.

\begin{definition}
\label{def_2}  Let $\,t\to \alpha (t)$ and $\,t\to \theta (t)$ be two continuous
functions such that $\alpha (t)\in (0,2]$ and $|\theta
(t)|\leq \min \{\alpha (t),2-\alpha (t)\}$. Then, for any $b>0,$ $t\geq 0,$%
\begin{equation}
\mathcal{P}_{b}^{\alpha (t),\theta (t)}:=-\ln (1-b\overline{\mathcal{D}^{\alpha
(t),\theta (t)}})= \int_{b^{-1}}^{\infty }\frac{1}{z} \overline{\mathcal{D}^{\alpha
(t),\theta (t)}}(z-\overline{\mathcal{D}^{\alpha (t),\theta (t)}})^{-1}dz.
\label{fr2}
\end{equation}%
\end{definition}
In the second case, instead, we allow  the parameter $b$ to be
time-dependent and thus we use the following operator.
\begin{definition}
\label{def_2bis}  Let
 $t\to b(t)$ be a continuous function such that $b(t)>0$ for any $t>0$. Moreover  let $\alpha \in (0,2],\;|\theta
|\leq \min \{\alpha ,2-\alpha \}$. Then, for any  $t\geq 0,$%
\begin{equation}
\mathcal{P}_{b(t)}^{\alpha ,\theta }:=-\ln (1-b(t)\overline{\mathcal{D}^{\alpha
,\theta }})=\int_{b(t)^{-1}}^{\infty }\frac{1}{z}\overline{\mathcal{D}^{\alpha
,\theta} }(z-\overline{\mathcal{D}^{\alpha ,\theta }})^{-1}dz. \label{shilling}
\end{equation}
\end{definition}
The symbols of the two operators defined above are
\begin{equation}
\widehat{\mathcal{P}_{b}^{\alpha (t),\theta (t)}}(\xi)=-\ln (1+b\psi _{\alpha(t) ,\theta(t) }(\xi )) \label{bla}
\end{equation}
and
\begin{equation}
\widehat{\mathcal{P}_{b(t)}^{\alpha,\theta}}(\xi)=-\ln (1+b(t)\psi _{\alpha ,\theta }(\xi ))\label{qua}
\end{equation}
respectively.\\

\section{First-type inhomogeneous GS process}

A natural way of defining a non-homogeneous generalization of the GS
process is by considering a time-varying
fractional index $\alpha (t),$ with $\alpha (t)\in (0,2],$ $\alpha (t)\neq
1, $ for any $t\geq 0,$ similarly to what is done for defining   the multistable
process in \cite{LEG}. This will influence the thickness of the  tails, which are
 no more  power-law as in the homogeneous case (see \cite{KOZ5}).  We also let
 the asymmetry parameter $\theta $ vary with $t$,
under the assumption that $|\theta (t)|\leq \min \{\alpha (t),2-\alpha
(t)\}, $ $\forall t\geq 0.$ By assigning the two functions $t \to \alpha (t)$
and $t\to \theta (t)$ with the above properties, and assuming their continuity,
 we have the following:

\begin{definition}
\label{def_3} \textbf{(Inhomogeneous GS process - I) }The process $G_{\alpha
,\theta }^{I}:=\left\{ G_{\alpha ,\theta }^{I}(t),t\geq 0\right\} $ is
defined by the following joint characteristic function:%
\begin{equation}
\mathbb{E}e^{i\sum_{j=1}^{d}\xi _{j}G_{\alpha ,\theta }^{I}(t_{j})}:=\exp
\left\{ -\int_{\mathbb{R}}\ln \left( 1+b\psi _{\alpha (z),\theta (z)}\left(
\sum_{j=1}^{d}\xi _{j}1_{[0,t_{j}]}(z)\right) \right) dz\right\} ,
\label{gss}
\end{equation}%
for $\xi _{j}\in \mathbb{R}$, $t_{j}\geq 0,\,d\in \mathbb{N}.$
\end{definition}

The process defined above (which, for brevity, we will call $GS^I$),
 is an additive process, as can be checked by
comparing (\ref{gss}) with (\ref{chh}) and taking into account the symbol of
the usual GS given in (\ref{ss}).

\begin{remark}
For the reader's convenience, we present the two main special cases of (\ref%
{gss}), i.e.

i) for $\alpha (t)\in (0,1)$, $\theta (t)=\pm \alpha (t)$ (completely
asymmetric case)%
\begin{equation}
\mathbb{E}e^{i\sum_{j=1}^{d}\xi _{j}G_{\alpha ,\pm \alpha }^{I}(t_{j})}:=\exp
\left\{ -\int_{\mathbb{R}}\ln \left( 1+b(\pm i)^{\alpha (z)}\biggl (
\sum_{j=1}^{d}\xi _{j}1_{[0,t_{j}]}(z)\biggr ) ^{\alpha (z)}\right)
dz\right\} .  \label{ii}
\end{equation}%
ii) for $\alpha (t)\in (0,2),$ with $\alpha (t)\neq 1$, and $\theta (t)=0$
(symmetric case)%
\begin{equation}
\mathbb{E}e^{i\sum_{j=1}^{d}\xi _{j}G_{\alpha ,0}^{I}(t_{j})}:=\exp \left\{
-\int_{\mathbb{R}}\ln \left( 1+b\left\vert \sum_{j=1}^{d}\xi
_{j}1_{[0,t_{j}]}(z)\right\vert ^{\alpha (z)}\right) dz\right\} .
\label{gs2}
\end{equation}%
It is easy to check the independence of increments: indeed, in the first
case, for $\theta (t)=\pm \alpha (t)$, we have, from (\ref{ii}), that%
\begin{eqnarray*}
\mathbb{E}e^{i\xi \lbrack G_{\alpha ,\pm \alpha }^{I}(t_{2})-G_{\alpha ,\pm
\alpha }^{I}(t_{1})]} &=&\exp \left\{ -\int_{\mathbb{R}}\ln \left( 1+b(\pm
i)^{\alpha (z)} \xi  ^{\alpha (z)}\left(
1_{[0,t_{2}]}(z)-1_{[0,t_{1}]}(z)\right) ^{\alpha (z)}\right) dz\right\} \\
&=&\exp \left\{ -\int_{\mathbb{R}}\ln \left( 1+b(\pm i)^{\alpha
(z)} \xi  ^{\alpha (z)}\left(
1_{[t_{1},t_{2}]}(z)\right) ^{\alpha (z)}\right) dz\right\} \\
&=&\exp \left\{ -\int_{t_{1}}^{t_{2}}\ln \left( 1+b(\pm i)^{\alpha
(z)} \xi  ^{\alpha (z)}\right) dz\right\}
\end{eqnarray*}%
and thus%
\begin{eqnarray*}
\mathbb{E}e^{i\xi G_{\alpha ,\pm \alpha }^{I}(t_{2})} &=&\exp \left\{
-\int_{t_{1}}^{t_{2}}\ln \left( 1+b(\pm i)^{\alpha (z)} \xi
 ^{\alpha (z)}\right) dz-\int_{0}^{t_{1}}\ln \left( 1+b(\pm
i)^{\alpha (z)} \xi  ^{\alpha (z)}\right) dz\right\} \\
&=&\mathbb{E}e^{i\xi \lbrack G_{\alpha ,\pm \alpha }^{I}(t_{2})-G_{\alpha
,-\alpha }^{I}(t_{1})]}\mathbb{E}e^{i\xi G_{\alpha ,\pm \alpha }^{I}(t_{1})}.
\end{eqnarray*}%
A similar check can be done for $\theta (t)=0.$
\end{remark}

\begin{theorem}
\label{th 1} Let \ $f\in \mathcal{S}(\mathbb{R})$ and let $\mathcal{P}%
_{b}^{\alpha (t),\theta (t)}$ be the operator defined in Def.\ref{def_2},
 then, for any $b>0,$ $t\geq s,$ the following initial value problem%
\begin{equation}
\left\{
\begin{array}{l}
\frac{\partial }{\partial t}u(x,t)=\mathcal{P}_{b,x}^{\alpha (t),\theta
(t)}u(x,t) \\
u(x,s)=f(x)%
\end{array}%
\right.  \label{gs3}
\end{equation}%
is satisfied by $T_{s,t}^{G_{\alpha
,\theta }^{I}}f(x):=\mathbb{E}\left[ \left. f\left( G_{\alpha ,\theta
}^{I}(t)\right) \right\vert G_{\alpha ,\theta }^{I}(s)=x\right] $.
\end{theorem}

\begin{proof}
The
process $G_{\alpha ,\theta }^{I}:=\left\{ G_{\alpha ,\theta }^{I}(t),t\geq
0\right\}$ has characteristic function
\begin{equation}
\mathbb{E}e^{i\xi G_{\alpha ,\theta }^{I}(t)}=\exp \left\{ -\int_{0}^{t}\ln
\left( 1+b\psi _{\alpha (z),\theta (z)}(\xi )\right) dz\right\} ,\qquad \xi
\in \mathbb{R}\text{, }t\geq 0,\,b>0.  \label{gs4}
\end{equation}
By taking the Fourier transform with respect to $x$, we can rewrite the
first equation in (\ref{gs3}) as follows:%
\begin{eqnarray*}
\frac{\partial }{\partial t}\widehat{u}(\xi ,t) &=&\widehat{\mathcal{P}%
_{b}^{\alpha (t),\theta (t)}}(\xi )\widehat{u}(\xi ,t) \\
&=&[\text{by (\ref{bla})}] \\
&=&-\ln (1+b\psi _{\alpha (t),\theta (t)}(\xi ))\widehat{u}(\xi ,t),
\end{eqnarray*}%
and Lemma \ref{lemma operatori pseudodifferenziali} (ii) gives the result.
\end{proof}

\begin{remark} \label{richiamo}
One could observe that (\ref{gs3})   does not apparently  coincide with the so called "forward"
equation for propagators  (consult, for example, \cite{RUS})
Indeed, by a simple calculation, we have
\begin{align*}
\frac{d}{dt} T_{s,t}f= \lim _{h\to 0^+} \frac{T_{s,t+h}-T_{s,t}}{h}f= \lim _{h\to 0^+}
\frac{T_{s,t}T_{t,t+h}-T_{s,t}}{h}f= \lim _{h\to 0^+} T_{s,t} \frac{T_{t,t+h}-I}{h}f= T_{s,t}\mathcal{A}_tf
\end{align*}
Since, in general, $T_{s,t}$ and $\mathcal{A}_t$ do not commute, it is not true
that $u(x,t)=T_{s,t}f(x)$ solves the non autonomous equation
$ \frac{d}{dt}u(x,t)= \mathcal{A}_tu(x,t)$ under $u(x,s)=f(x)$. However,
in our case, $f$ lives in the Schwartz space, and it is easy to
check that $T_{s,t}$ and $\mathcal{A}_t$ (which are given by
(\ref{rappresentazione pseudodifferenziale del propagatore})
and (\ref{rappresentazione pseudodifferenziale del generatore})), commute.
\end{remark}

\subsubsection{On the subordinator}
In the special case where $\theta (t)=-\alpha (t)$ and $\alpha (t)\in (0,1)$
for any $t>0$, we can easily evaluate the L\'{e}vy measure of the process $G_{\alpha ,-\alpha }^{I}$,
which is a inhomogeneous subordinator in the sense of \cite{ORT}.
Recall that $\psi _{\alpha
(t),-\alpha (t)}(\xi )=(-i\xi )^{\alpha (t)}$ and thus the Laplace transform
of $G_{\alpha ,-\alpha }^{I}$ can be written as follows%
\begin{equation}
\mathbb{E}e^{-\lambda G^I_{\alpha, -\alpha}(t)} =exp \biggl ( -\int_{0}^{t}\ln (1+b\lambda
^{\alpha (s)})ds \biggr ) ,\qquad \lambda >0.  \label{pr}
\end{equation}%
The integral in (\ref{pr}) is finite, since%
\begin{equation*}
\int_{0}^{t}\ln (1+b\lambda ^{\alpha (s)})ds\leq b\int_{0}^{t}\lambda
^{\alpha (s)}ds<\infty ,
\end{equation*}%
for $\alpha (s)\in (0,1),$ for any $s.$

\begin{lemma}
\label{lm_1} The time dependent L\'{e}vy measure of $G_{\alpha ,-\alpha }^{I}$ is given by%
\begin{equation}
\nu _{t}^{G_{\alpha ,-\alpha }^{I}}(dx)=x^{-1}\alpha (t)E_{\alpha
(t)}(-x^{\alpha (t)}/b)dx,  \label{lm}
\end{equation}%
where, for any $s\geq 0$,  $E_{\alpha (s)}(x)$ denotes the Mittag-Leffler function $E_{\alpha
(s)}(x)=\sum_{j=0}^{\infty }\frac{x^{j}}{\Gamma (\alpha (s)j+1)}.$
\end{lemma}

\begin{proof}
From (\ref{lm}), by applying formula (1.9.13), p.47 in \cite{KIL}, for any
fixed $t,$ we can write%
\begin{eqnarray*}
&&\int_{0}^{+\infty }(e^{-\lambda x}-1)\nu _{t}^{G_{\alpha ,-\alpha
}^{I}}(x)dx \\
&=-&\alpha (t)\int_{0}^{+\infty }\left( \int_{0}^{\lambda }e^{-zx}dz\right)
E_{\alpha (t)}(-x^{\alpha (t)}/b)dx \\
&=-&\alpha (t)\int_{0}^{\lambda }\left( \int_{0}^{+\infty }e^{-zx}E_{\alpha
(t)}(-x^{\alpha (t)}/b)dx\right) dz \\
&=-&b\alpha (t)\int_{0}^{\lambda }\frac{z^{\alpha (t)-1}}{bz^{\alpha (t)}+1}dz\\
&=&-\ln (1+b\lambda ^{\alpha (t)}),
\end{eqnarray*}%
which agrees with (\ref{pr}). We now check that the condition $%
\int_{0}^{+\infty }\frac{x}{1+x}\nu _{t}^{G_{\alpha ,-\alpha
}^{I}}(x)dx<\infty $ is satisfied, as follows%
\begin{eqnarray*}
&&\int_{0}^{+\infty }\frac{x}{1+x}\nu _{t}^{\mathcal{G}_{\alpha ,-\alpha
}^{I}}(x)dx \\
&=&\alpha (t)\int_{0}^{+\infty }\frac{1}{1+x}E_{\alpha (t)}(-x^{\alpha
(t)}/b)dx \\
&=&[\text{by (3.4.30) in \cite{GOR}}] \\
&\leq &b\alpha (t)M^{+}\int_{0}^{+\infty }\frac{1}{1+x}\frac{1}{b+x^{\alpha
(t)}}dx <\infty ,
\end{eqnarray*}%
where $M^{+}$ is a positive constant.
\end{proof}

It is easy to check that, letting $\alpha (t)\rightarrow 1$, for any $t,$ we
obtain from (\ref{lm}) that
\begin{equation*}
\lim_{\alpha (t)\rightarrow 1}\nu_{t}^{G_{\alpha ,-\alpha }^{I}}(x)dx=x^{-1}e^{-x/b}=\nu _{\mathit{\Gamma }%
}(dx)\notag
\end{equation*}
while, for $\alpha (t)=\alpha $, for any $t\geq 0$, we get the L\'{e}%
vy measure of the standard GS process, which reads $\nu _{t}^{G_{\alpha
,-\alpha }}(x)dx=x^{-1}\alpha E_{\alpha }(-x^{\alpha }/b)$ (see \cite{SVON}).

Finally, we show how the tails' behavior of the density of the process $%
G_{\alpha ,-\alpha }^{I}$, for any fixed $t$, differs from those holding for
both the stable and geometric stable random variables (see \cite{SAMO}, p.
17, and \cite{KOZP} respectively). The latter is, obviously, obtained in the
special case where $\alpha (s)=\alpha ,$ for any $s.$

\begin{theorem}
Let $\alpha _{t}^{\ast }:=\max_{0\leq s\leq t}\alpha (s),$ then, for $x\rightarrow \infty ,$%
\begin{equation}
P(G_{\alpha ,-\alpha }^{I}(t)>x)\sim \frac{b}{\Gamma (1-\alpha _{t}^{\ast })}%
\int_{0}^{t}\alpha (s)x^{-\alpha (s)}ds,\qquad t\geq 0.  \label{tau}
\end{equation}
\end{theorem}

\begin{proof}
By (\ref{pr}), we can write that%
\begin{eqnarray}
\int_{0}^{+\infty }e^{-\eta x}P(G_{\alpha ,-\alpha }^{I}(t)>x)dx &=&\frac{1-%
\mathbb{E}e^{-\eta \mathcal{G}_{\alpha ,-\alpha }^{I}(t)}}{\eta }  \label{tt}
\\
&=&\frac{1-\exp \{-\int_{0}^{t}\ln (1+b\eta ^{\alpha (s)})ds\}}{\eta }
\notag \\
&\sim &b\eta ^{\alpha _{t}^{\ast }-1}\int_{0}^{t}\alpha (s)\eta
^{\alpha (s)-\alpha _{t}^{\ast }} ds,  \notag
\end{eqnarray}%
for any fixed $t$ and for $\eta \rightarrow 0.$ The integral in the last
line of (\ref{tt}) is a regularly varying function in $1/\eta $, since, for
any real $k,$ by the mean value theorem we have that%
\begin{equation*}
\frac{\int_{0}^{t}\alpha (s)(k\eta )^{\alpha (s)-\alpha _{t}^{\ast }}ds}{%
\int_{0}^{t}\alpha (s)\eta ^{\alpha (s)-\alpha _{t}^{\ast }}ds}=
k^{\alpha (\overline{s}_{t})-\alpha _{t}^{\ast }},
\end{equation*}%
where $\overline{s}_{t}\in (0,t]$. Then, by applying the Tauberian theorem
(see Theorem XIII-5-4, p.446, in \cite{FEL}), as $x\to \infty$ we obtain%
\begin{equation*}
P(G_{\alpha ,-\alpha }^{I}(t)>x)\sim \frac{bx^{-\alpha _{t}^{\ast }}}{\Gamma
(1-\alpha _{t}^{\ast })}\int_{0}^{t}\alpha (s)x^{\alpha _{t}^{\ast }-\alpha
(s)}ds.
\end{equation*}
\end{proof}

\begin{remark}
An analogous result is proved to hold, in \cite{AYA}, for the multistable
symmetric process.
\end{remark}

\section{Second-type inhomogeneous GS process}

We here recall the definition of  time-inhomogeneous (or additive)
Gamma subordinator, which we denote by $\mathit{\Gamma }%
^{I}:=\left\{ \mathit{\Gamma }^{I}(t),t\geq 0\right\} $. The latter
has been studied for the first time in \cite{CIN} and then considered
in \cite{ORT} as a remarkable example among inhomogeneous subordinators.
 For an assigned strictly positive and bounded function $s\to b(s)$, the process $\mathit{\Gamma }%
^{I}$ is completely determined by its finite dimensional distributions
\begin{align} \label{ggg}
\mathbb{E}e^{i\sum_{j=1}^{d}\xi _{j}\mathit{\Gamma }^{I}(t_{j})}:=\exp
\left\{ -\int_{\mathbb{R}}\ln \left( 1-ib(s) \sum_{j=1}^{d}\xi
_{j}1_{(0,t_{j})}(s) \right) ds\right\} ,\qquad \text{ }
\end{align}%
corresponding to the time-dependent L\'{e}vy density
\begin{equation}
\nu _{t}^{\mathit{\Gamma }^{I}}(x)=x^{-1}e^{-x/b(t)}, \qquad x>0. \label{mu}
\end{equation}%
 It is evident that, in the special case $b(t)=b$, the process $\mathit{\Gamma }^{I}$ reduces to the standard
Gamma subordinator $\mathit{\Gamma }.$

We prove now the following result concerning the governing equation of the
additive Gamma subordinator, by considering that, for $\alpha =1, \theta=-1$,
equations (\ref{shilling}) and (\ref{qua}) reduce to
\begin{equation}
\mathcal{P}_{b(t),x}f(x)=-\ln \left(1-b(t)\frac{d}{d x}\right)f(x)
=\int_{b(t)^{-1}}^{\infty }\frac{1}{z}\frac{d}{d x}\left(z-\frac{d}{d x}\right)^{-1}f(x)dz \notag
\end{equation}
and
\begin{equation}
\widehat{\mathcal{P}_{b(t)}}(\xi)=-\ln (1-ib(t)\xi )\label{symb}
\end{equation}
respectively.

\begin{lemma}
\label{lm_2} Let $f\in \mathcal{S}(\mathbb{R})$. Then  the propagator $T_{s,t}^{%
\mathit{\Gamma }^{I}}f(x):=\mathbb{E}\left[ \left. f\left( \mathit{\Gamma }%
^{I}(t)\right) \right\vert \mathit{\Gamma }^{I}(s)=x\right] $ associated to the process
$\mathit{\Gamma }^{I}$ satisfies the following initial value problem%
\begin{equation*}
\left\{
\begin{array}{l}
\frac{\partial }{\partial t}u(x,t)=-\ln \left( 1-b(t)\frac{\partial }{%
\partial x}\right) u(x,t), \qquad t \geq s \\
u(x,s)=f(x)%
\end{array}%
\right. .
\end{equation*}
where the operator on the right side must be meant in the sense of (\ref{generatore logaritmico}).
\end{lemma}

\begin{proof}
We take the first time-derivative of (\ref{ggg}), in the case $d=1$ with $t_{1}=t$
and $\xi_{1}=\xi$, so that we get%
\begin{eqnarray*}
\frac{\partial }{\partial t}\mathbb{E}e^{i\xi \mathit{\Gamma }^{I}(t)}
&=&-\ln \left( 1-ib(t)\xi \right) \mathbb{E}e^{i\xi \mathit{\Gamma }^{I}(t)}
\\
&=&\widehat{\mathcal{P}_{b(t)}}(\xi )\mathbb{E}e^{i\xi \mathit{\Gamma }%
^{I}(t)},
\end{eqnarray*}%
and  considering Lemma \ref{lemma operatori pseudodifferenziali} (ii), the proof is complete.
\end{proof}

In order to let this process verify the useful property of finite
exponential moments, we make the further assumption that $b(t)<K,$ for any $%
t\geq 0$ and for a constant $K<1.$ Thus, for any fixed $t$ and for $|u|\leq
\frac{1}{K},$  we have that%
\begin{equation}
\int_{0}^{+\infty }e^{ux}\mu _{t}^{\mathit{\Gamma }^{I}}(x)dx=\int_{0}^{t}%
\int_{0}^{+\infty }x^{-1}e^{x(u-1/b(s))}dxds<\infty.  \label{em}
\end{equation}%
 This is a necessary and
sufficient condition for the finiteness of the moment generating function
(see \cite{EBE2}), which, thus, is finite for any $|\gamma |\leq \frac{1}{k}$, and
reads%
\begin{equation*}
\mathbb{E}e^{\gamma \mathit{\Gamma }^{I}(t)}=\exp \left\{ -\int_{0}^{t}\ln
\left( 1-\gamma b(s)\right) ds\right\} ,
\end{equation*}%
so that we get%
\begin{eqnarray}
\mathbb{E}\mathit{\Gamma }^{I}(t) &=&\int_{0}^{t}b(s)ds  \label{mom} \\
\mathbb{V}\left[ \mathit{\Gamma }^{I}(t)\right] &=&\int_{0}^{t}b(s)^{2}ds
\notag
\end{eqnarray}%
As a consequence of (\ref{em}), we can state that the additive Gamma process
$\mathit{\Gamma }^{I}$ is a special semimartingale (see \cite{JAC}) and then
it is suitable for financial applications (see \cite{KAL}), when $b(t)<K<1.$

The additive Gamma process is the fundamental ingredient to construct the following.

\begin{definition}
\label{def_4} \textbf{(Inhomogeneous GS process - II) }Let $\ \left\{
S_{\alpha ,\theta }(t),t\geq 0\right\} $ be an $\alpha $-stable process defined
in (\ref{funzione caratteristica processo stabile}) and $%
\mathit{\Gamma }^{I}$ be a inhomogeneous Gamma subordinator, independent from $%
S_{\alpha ,\theta },$   we
define $G_{\alpha ,\theta }^{II}:=\left\{ G_{\alpha ,\theta }^{II}(t),t\geq
0\right\} $ by the following subordination (see section 2.2)
\begin{equation}
G_{\alpha ,\theta }^{II}(t):=S_{\alpha ,\theta }(\mathit{\Gamma }%
^{I}(t)),\qquad t\geq 0.  \label{sub7}
\end{equation}%
\end{definition}

\begin{theorem}
\label{th_3} Let $f\in S(\mathbb{R})$, then $T_{s,t}^{%
\mathcal{G}_{\alpha ,\theta }^{II}}f(x):=\mathbb{E}\left[ \left. f\left(
G_{\alpha ,\theta }^{II}(t)\right) \right\vert G_{\alpha ,\theta }^{II}(s)=x%
\right] $ satisfies the following
initial value problem%
\begin{equation}
\left\{
\begin{array}{l}
\frac{\partial }{\partial t}u(x,t)=\mathcal{P}_{b(t),x}^{\alpha ,\theta
}u(x,t), \qquad t \geq s \\
u(x,s)=f(x).%
\end{array}%
\right. ,  \label{pr4}
\end{equation}%
where the operator $\mathcal{P}_{b(t),x}^{\alpha ,\theta }$ is defined in
Def.\ref{def_2bis}.
\end{theorem}

\begin{proof}
For any $t>0$, we can evaluate the characteristic function of $S_{\alpha ,\theta }(\mathit{%
\Gamma }^{I}(t)),$ by a standard conditioning argument
\begin{eqnarray}
\mathbb{E}e^{i\xi S_{\alpha ,\theta }(\mathit{\Gamma }^{I}(t))} &=&\mathbb{E}%
\left\{ \mathbb{E}\left[ \left. e^{i\xi S_{\alpha ,\theta }(\mathit{\Gamma }%
^{I}(t))}\right\vert \mathit{\Gamma }^{I}(t)\right] \right\} =\mathbb{E}\exp
\{-\psi _{\alpha ,\theta }(\xi )\mathit{\Gamma }^{I}(t)\}  \label{sub8} \\
&=&[\text{by (\ref{funzione caratteristica processo stabile})}]  \notag \\
&=&\exp \left\{ -\int_{0}^{t}\ln \left[ 1+b(s)\psi _{\alpha ,\theta }(\xi )%
\right] ds\right\} ,  \notag
\end{eqnarray}%
where $b(s)\geq 0$, for any $s\geq 0$ and $\int_{0}^{t}b(s)ds<\infty .$
By taking the Fourier transform of the first equation in (\ref{pr4}), we get,
by (\ref{qua})%
\begin{eqnarray*}
\frac{\partial }{\partial t}\widehat{u}(\xi ,t) &=&\widehat{\mathcal{P}%
_{b(t)}^{\alpha ,\theta }}(\xi )\widehat{u}(\xi ,t) \\
&=&-\ln (1+b(t)\psi _{\alpha ,\theta }(\xi ))\widehat{u}(\xi ,t)
\end{eqnarray*}%
and Lemma \ref{lemma operatori pseudodifferenziali} (ii) gives the result.
\end{proof}

 Of course, the same considerations made in Remark \ref{richiamo} are also valid for Theorem \ref{th_3}.

\begin{remark}
We observe
that the process defined in Def (\ref{def_4})
 can be obtained by a homogeneous GS by a time-dependent
change of scale, and this makes  $GS^{II}$ statistically tractable, as opposed
to the usual case of non-stationary processes.
Indeed, let $G_{\alpha, \theta}$ be a homogeneous geometric stable process such that
\begin{align*}
\mathbb{E}e^{i \xi G_{\alpha, \theta}(t)}= e^{-t \ln (1+\psi _{\alpha, \theta}(\xi))}
\end{align*}
We divide the interval $[0,t]$ into $n$ subintervals of length
$t/n$ and write the telescopic series $G_{\alpha, \theta}(t) =
\sum _{k=0}^{n-1} (G_{\alpha, \theta}(t_{k+1})-G_{\alpha,
\theta}(t_k))$. In order to change the scale, we let the increment
$G_{\alpha, \theta}(t_{k+1})-G_{\alpha, \theta}(t_k)$ be changed
into $b(t_k)^{\frac{1}{\alpha}}\bigl (  G_{\alpha,
\theta}(t_{k+1})-G_{\alpha, \theta}(t_k)    \bigr )$. Then, for
any $t\geq 0$, the following limit as $n\to \infty$ holds in
distribution
\begin{align}\label{trasformazione di scala}
\sum _{k=0}^{n-1} b(t_k)^{\frac{1}{\alpha}}\bigl (  G_{\alpha, \theta}
(t_{k+1})-G_{\alpha, \theta}(t_k)    \bigr )\stackrel{d}{\to}G^{II}_{\alpha, \theta}(t)
\end{align}
Indeed, by independence of the increments,
\begin{align*}
& \mathbb{E}\, \textrm{exp} \biggl \{ i \xi \sum _{k=0}^{n-1} b(t_k)^{\frac{1}{\alpha}}
  \bigl ( G_{\alpha, \theta}(t_{k+1})-G_{\alpha, \theta}(t_k)   \bigr )    \biggr \}\\
&= \prod _{k=0}^{n-1} \mathbb{E}\,\, exp \biggl \{i \xi b(t_k)^{\frac{1}{\alpha}}
(G_{\alpha, \theta}(t_{k+1})-G_{\alpha, \theta}(t_k))\biggr \}\\
&= \prod _{k=0}^{n-1} exp \biggl \{ -(t_{k+1}-t_k)\ln (1+b(t_k)\psi _{\alpha, \theta}(\xi))  \biggr \} \\
& \stackrel{n\to \infty}{\to}\, \,  exp \biggl \{-\int _0^t \ln (1+b(\tau)
 \psi _{\alpha, \theta} (\xi))  d\tau \biggr \}.
\end{align*}
\end{remark}

We derive now the tails' behavior of the density of the process $G_{\alpha
,\theta }^{II}$, for any fixed $t$. In this case, contrary to $GS^{I}$,
 we can prove the following result in the more general setting,
i.e. for any $\alpha \in (0,2]$ and $|\theta| \leq \min \{\alpha ,2-\alpha \}$%
, not only in the completely positively skewed case. Indeed, we can resort
here to the subordinating relation (\ref{sub7}).

\begin{theorem}
Let $\alpha \in (0,2]$ and $|\theta |\leq \min \{\alpha ,2-\alpha \},$ then
we have that
\begin{equation}
\left\{
\begin{array}{l}
\lim _{x\to \infty}x^\alpha P(G_{\alpha ,\theta }^{II}(t)>x)= \frac{C_{\alpha ,\theta }}{%
\Gamma (1-\alpha )}\int_{0}^{t}b(s)ds \\
\lim_{x\to \infty} x^\alpha P(G_{\alpha ,\theta }^{II}(t)<-x)= \frac{\overline{C}_{\alpha ,\theta
}}{\Gamma (1-\alpha )}\int_{0}^{t}b(s)ds%
\end{array}%
\right. ,\qquad t\geq 0,  \label{re}
\end{equation}%
where $C_{\alpha ,\theta }=\frac{1}{2}\left[ 1-\frac{\tan (\pi \theta /2)}{%
\tan (\pi \alpha /2)}\right] $ and $\overline{C}_{\alpha ,\theta }=\frac{1}{2%
}\left[ 1+\frac{\tan (\pi \theta /2)}{\tan (\pi \alpha /2)}\right] .$
\end{theorem}

\begin{proof}
We prove the first relation in (\ref{re}), the second one can be obtained by
similar arguments. We apply Property 1.2.15 in \cite{SAMO}, p.16, regarding
 the tail behavior of the stable process, which in our parametrization
of stable laws \footnote{In agreement with part of the literature,
 the authors in \cite{SAMO} express the characteristic function
 of stable distributions using the parameters triplet $\alpha, \sigma, \beta $
 (see \cite{SAMO} p.7). Our parametrization with
 $(\alpha, \theta)$, due to Feller, can be obtained by setting
 $\sigma= (\cos \frac{\pi \theta}{2})^{\frac{1}{\alpha}}$  and
 $\theta= \frac{2}{\pi}\arctan (-\beta \tan\frac{\pi \alpha}{2})$. } reads
\begin{equation*}
\lim_{x\rightarrow \infty }x^{\alpha }P(S_{\alpha ,\theta }(t)>x)=\frac{%
C_{\alpha ,\theta }\,t}{\Gamma (1-\alpha )}.
\end{equation*}%
By Def.\ref{def_4}, we have
\begin{align*}
\lim_{x\rightarrow \infty }x^{\alpha }P(S_{\alpha ,\theta }(\Gamma (t))>x)&
=\lim_{x\rightarrow \infty }x^{\alpha }\int_{0}^{\infty }P(S_{\alpha ,\theta
}(\mathit{\Gamma }^{I}(t))>x|\Gamma ^I(t)=z)P(\mathit{\Gamma }^{I}(t)\in dz) \\
& =\int_{0}^{\infty }\lim_{x\rightarrow \infty }x^{\alpha }P(S_{\alpha
,\theta }(z)>x)P(\mathit{\Gamma }^{I}(t)\in dz) \\
& =\frac{C_{\alpha ,\theta }}{\Gamma (1-\alpha )}\mathbb{E}\mathit{\Gamma }%
^{I}(t),
\end{align*}%
which gives (\ref{re}), by (\ref{mom}).
\end{proof}

\subsubsection{Computation of the L\'{e}vy measure}

We now compute explicitly the time-dependent L\'{e}vy measure of the
$GS^{II}$ process. This expression is not even known
for the homogeneous geometric stable process, with the exception of the case of the GS
subordinator (see, for example, \cite{SVON}).
 In the following, let ${\LARGE H}_{p,q}^{m,n}$ denote the H-function defined as
(see \cite{MAT} p.13):%
\begin{equation*}
{\LARGE H}_{p,q}^{m,n}\left[ \left. z\right\vert
\begin{array}{ccc}
(a_{1},A_{1}) & ... & (a_{p},A_{p}) \\
(b_{1},B_{1}) & ... & (b_{q},B_{q})%
\end{array}%
\right] =\frac{1}{2\pi i}\int_{\mathcal{L}}\frac{\left\{
\prod\limits_{j=1}^{m}\Gamma (b_{j}+B_{j}s)\right\} \left\{
\prod\limits_{j=1}^{n}\Gamma (1-a_{j}-A_{j}s)\right\} z^{-s}ds}{\left\{
\prod\limits_{j=m+1}^{q}\Gamma (1-b_{j}-B_{j}s)\right\} \left\{
\prod\limits_{j=n+1}^{p}\Gamma (a_{j}+A_{j}s)\right\} },
\end{equation*}%
with $z\neq 0,$ $m,n,p,q\in \mathbb{N}_{0},$ for $0\leq m\leq q$, $0\leq
n\leq p$, $a_{j},b_{j}\in \mathbb{R},$ $A_{j},B_{j}\in \mathbb{R}^{+},$ for $%
i=1,...,p,$ $j=1,...,q$ and $\mathcal{L}$ is a contour such that the
following condition is satisfied%
\begin{equation}
A_{\lambda }(b_{j}+\nu )\neq B_{j}(a_{\lambda }-k-1),\qquad j=1,...,m,\text{
}\lambda =1,...,n,\text{ }\nu ,k=0,1,...  \label{1.6}
\end{equation}

Let moreover $\mu =\sum_{j=1}^{q}B_{j}-\sum_{j=1}^{p}A_{j}.$\\
For details on the use of H-function in connection with fractional process, see  \cite{MAI2}.

\begin{lemma}
\label{lm_4}
The L\'{e}vy measure of $GS^{II}$ defined in Def.\ref{def_4} is given, for $%
\alpha \in (0,1),$ by%
\begin{equation}
\nu _{t}^{G_{\alpha ,\theta }^{II}}(dx)=\alpha \frac{dx}{|x|}{\LARGE H}%
_{3,2}^{1,2}\left[ \left. \frac{b(t)}{|x|^{\alpha }}\right\vert
\begin{array}{ccc}
(1,\alpha ) & (1,1) & \left( 1,\frac{\alpha -\theta sign(x)}{2}\right)  \\
(1,1) & \left( 1,\frac{\alpha -\theta sign(x)}{2}\right)  & \;%
\end{array}%
\right] ,\text{ }x\neq 0,\;|\theta |\leq \alpha   \label{hh}
\end{equation}%
while, for $\alpha \in (1,2)$, by%
\begin{equation}
\nu _{t}^{G_{\alpha ,\theta }^{II}}(dx)=\frac{dx}{\alpha |x|}{\LARGE H}%
_{2,3}^{2,1}\left[ \left. \frac{|x|}{b(t)^{1/\alpha }}\right\vert
\begin{array}{ccc}
\left( 1,\frac{1}{\alpha }\right)  & \left( 1,\frac{\alpha -\theta sign(x)}{%
2\alpha }\right)  & \; \\
(1,1) & \left( 0,\frac{1}{\alpha }\right)  & \left( 1,\frac{\alpha
-\theta sign(x)}{2\alpha }\right)
\end{array}%
\right] ,\text{ }x\neq 0,\;|\theta |\leq 2-\alpha .  \label{hh2}
\end{equation}
\end{lemma}

\begin{proof}
We apply the series representation of the stable law (see
\cite{FEL}, Lemma 1, p.583), together with the reflection property
of the Gamma function. Then, by considering the subordinating
relationship $S_{\alpha ,\theta }(\Gamma ^{I})$ together with
(\ref{mu}) and (1.140) in \cite{MAT},
we can write, for $x>0$ and $%
\alpha \in (0,1)$, with $|\theta |\leq \alpha ,$%
\begin{eqnarray*}
\nu _{t}^{G_{\alpha ,\theta }^{II}}(dx) &=&dx\int_{0}^{+\infty
}p_{\alpha
}(x;\theta ,\tau )\tau ^{-1}e^{-\tau /b(t)}d\tau \\
&=&dx\int_{0}^{+\infty }\frac{1}{\tau ^{1/\alpha }}p_{\alpha }\left( \frac{x%
}{\tau ^{1/\alpha }};\theta ,1\right) \tau ^{-1}e^{-\tau /b(t)}d\tau \\
&=&\frac{dx}{\pi x}\sum_{k=1}^{\infty }(-x^{-\alpha
})^{k}\frac{\Gamma (k\alpha +1)\sin (k(\theta -\alpha )\pi
/2)}{k!}\int_{0}^{+\infty }\tau
^{k-1}e^{-\tau /b(t)}d\tau \\
&=&-\frac{\alpha dx}{\pi x}\sum_{k=1}^{\infty }\frac{(-b(t)x^{-\alpha })^{k}%
}{(k-1)!}\Gamma (k\alpha )\Gamma (k)\sin (k(\alpha -\theta )\pi /2) \\
&=&-\frac{\alpha dx}{x}\sum_{k=1}^{\infty }\frac{(-b(t)x^{-\alpha })^{k}}{%
(k-1)!}\frac{\Gamma (k\alpha )\Gamma (k)}{\Gamma (k(\alpha -\theta
)/2)\Gamma (1-k(\alpha -\theta )/2)} \\
&=&\alpha dx b(t)x^{-\alpha -1}\sum_{l=0}^{\infty
}\frac{(-b(t)x^{-\alpha })^{l}}{l!}\frac{\Gamma (l\alpha +\alpha
)\Gamma (l+1)}{\Gamma ((l+1)(\alpha
-\theta )/2)\Gamma (1-(l+1)(\alpha -\theta )/2)} \\
&=&\frac{\alpha dx b(t)x^{-\alpha -1}}{2\pi i}\int_{L}\frac{\Gamma
(s)\Gamma (\alpha -\alpha s)\Gamma (1-s)(b(t)x^{-\alpha })^{-s} ds
}{\Gamma ((\alpha
-\theta )(1-s)/2)\Gamma (1-(\alpha -\theta )(1-s)/2)} \\
&=&\alpha dx b(t)x^{-\alpha -1}{\LARGE H}_{3,2}^{1,2}\left[ \left. \frac{b(t)%
}{x^{\alpha }}\right\vert
\begin{array}{ccc}
(1-\alpha ,\alpha ) & (0,1) & \left( 1-\frac{\alpha -\theta }{2},\frac{%
\alpha -\theta }{2}\right) \\
(0,1) & \left( 1-\frac{\alpha -\theta }{2},\frac{\alpha -\theta
}{2}\right)
& \;%
\end{array}%
\right]
\end{eqnarray*}%
which coincides with (\ref{hh}), by applying (1.60) of \cite{MAT}, with $%
\sigma =1.$ It is immediate to check that condition (\ref{1.6}) is
satisfied. As far as the case $\alpha \in (1,2)$ with $|\theta
|\leq 2-\alpha ,$ is concerned, we have that

\begin{eqnarray}
\nu _{t}^{G_{\alpha ,\theta }^{II}}(dx) &=&dx\int_{0}^{+\infty }\frac{1}{%
\tau ^{1/\alpha }}p_{\alpha }\left( \frac{x}{\tau ^{1/\alpha
}};\theta
,1\right) \tau ^{-1}e^{-\tau /b(t)}d\tau   \label{hh3} \\
&=&\frac{dx}{\pi x}\sum_{k=1}^{\infty }(-x)^{k}\frac{\Gamma
(k/\alpha
+1)\sin (k(\theta -\alpha )\pi /2\alpha )}{k!}\int_{0}^{+\infty }\tau ^{-%
\frac{k}{\alpha }-1-\frac{1}{\alpha }}e^{-\tau /b(t)}d\tau   \nonumber \\
&=&-\frac{dx}{\alpha \pi x}\sum_{k=1}^{\infty
}\frac{(-b(t)^{-1/\alpha }x)^{k}}{(k-1)!}\Gamma (k/\alpha )\Gamma
(-k/\alpha )\sin (k(\alpha -\theta
)\pi /2\alpha )  \nonumber \\
&=&-\frac{dx}{\alpha x}\sum_{k=1}^{\infty }\frac{(-b(t)^{-1/\alpha }x)^{k}}{%
(k-1)!}\frac{\Gamma (k/\alpha )\Gamma (-k/\alpha )}{\Gamma
(k(\alpha -\theta
)/2\alpha )\Gamma (1-k(\alpha -\theta )/2\alpha )}  \nonumber \\
&=&\frac{dx b(t)^{-1/\alpha }}{\alpha }\sum_{l=0}^{\infty }\frac{%
(-b(t)^{-1/\alpha }x)^{l}}{l!}\frac{\Gamma ((l+1)/\alpha )\Gamma
(-(l+1)/\alpha )}{\Gamma ((l+1)(\alpha -\theta )/2\alpha )\Gamma
(1-(l+1)(\alpha -\theta )/2\alpha )}  \nonumber \\
&=&\frac{dx b(t)^{-1/\alpha }}{2\alpha \pi
i}\int_{\mathcal{L}}\frac{\Gamma
(s)\Gamma ((1-s)/\alpha )\Gamma ((s-1)/\alpha )(b(t)^{-1/\alpha }x)^{-s}ds}{%
\Gamma ((\alpha -\theta )(1-s)/2\alpha )\Gamma (1-(\alpha -\theta
)(1-s)/2\alpha )}  \nonumber \\
&=&\frac{dx}{\alpha b(t)^{1/\alpha }}{\LARGE H}_{2,3}^{2,1}\left[
\left. \frac{|x|}{b(t)^{1/\alpha }}\right\vert
\begin{array}{ccc}
\left( 1-\frac{1}{\alpha },\frac{1}{\alpha }\right)  & \left(
\frac{\alpha
+\theta }{2\alpha },\frac{\alpha -\theta }{2\alpha }\right)  & \; \\
(0,1) & \left( -\frac{1}{\alpha },\frac{1}{\alpha }\right)  & \left( \frac{%
\alpha +\theta }{2\alpha },\frac{\alpha -\theta }{2\alpha }\right)
\nonumber
\end{array}%
\right] ,
\end{eqnarray}%
which, by considering that $\alpha >\theta ,$ coincides with
(\ref{hh2}). For $x<0$ similar steps lead to (\ref{hh}) and
(\ref{hh2}), respectively, for $\alpha \in (0,1)$ and $\alpha \in
(1,2),$ by applying formula (6.4) in \cite{FEL}.
\end{proof}

\begin{remark}
We can check that (\ref{hh}) coincides with the L\'{e}vy measure
of the GS subordinator, for $\theta =-\alpha $, $b(t)=b,$ for any
$t,$ and $\alpha \in (0,1)$ (see \cite{BUR}). Indeed, in this
case, we can write (\ref{hh}) as
follows, by applying (1.56), (1.58) and (1.48) in \cite{MAT}:%
\begin{eqnarray*}
\nu _{t}^{G_{\alpha }^{II}}(dx) &=&\frac{\alpha }{x}dx{\LARGE H}_{2,1}^{1,1}%
\left[ \left. \frac{b(t)}{x^{\alpha }}\right\vert
\begin{array}{cc}
(1,1) & (1,\alpha ) \\
\left( 1,1\right)  & \;%
\end{array}%
\right]  \\
&=&\frac{\alpha }{x}dx{\LARGE H}_{1,2}^{1,1}\left[ \left. \frac{x^{\alpha }}{%
b(t)}\right\vert
\begin{array}{cc}
(0,1) & \; \\
\left( 0,1\right)  & (0,\alpha )%
\end{array}%
\right]  \\
&=&\alpha x^{-1}E_{\alpha }(-x^{\alpha }/b(t))dx,
\end{eqnarray*}%
where $E_{\alpha }(z)$ denotes the Mittag-Leffler function, for
$\alpha >0,$
$z\in \mathbb{C}.$ Moreover, by letting $\alpha \rightarrow 1^{-},$ we get (%
\ref{mu}).
\end{remark}

\begin{remark}
On the other hand, we can check that (\ref{hh2}), in the special case $%
\alpha =2$, $b(t)=b,$ for any $t,$ and $\theta =0$, coincides with the L\'{e}%
vy measure of the VG process (see formula (13) in \cite{MAD}).
Indeed, from the last line of (\ref{hh3}) and taking into account
(1.60), for $\sigma =-1$,
and (1.37)-(1.38) in \cite{MAT}, we have that%
\begin{eqnarray*}
\nu _{t}^{G_{\alpha ,\theta }^{II}}(dx) &=&\frac{dx}{2|x|}{\LARGE H}%
_{1,2}^{2,0}\left[ \left. \frac{|x|}{\sqrt{b}}\right\vert
\begin{array}{cc}
\left( 1,\frac{1}{2}\right)  & \, \\
\left( 1,1\right)  & \left( 0,\frac{1}{2}\right)
\end{array}%
\right]  \\
&=&\frac{dx}{2\sqrt{b}}{\LARGE H}_{1,2}^{2,0}\left[ \left. \frac{|x|}{\sqrt{b%
}}\right\vert
\begin{array}{cc}
\left( \frac{1}{2},\frac{1}{2}\right)  & \; \\
\left( 0,1\right)  & \left( -\frac{1}{2},\frac{1}{2}\right)
\end{array}%
\right]  \\
&=&\frac{dx}{2\sqrt{b}}\frac{1}{2\pi i}\int_{L}\frac{\Gamma
(s)\Gamma \left(
\frac{s}{2}-\frac{1}{2}\right) \left( \frac{|x|}{\sqrt{b}}\right) ^{-s}ds}{%
\Gamma \left( \frac{s}{2}+\frac{1}{2}\right) } \\
&=&\frac{dx}{\sqrt{b}}\frac{1}{2\pi i}\int_{L}\frac{\Gamma (s)\left( \frac{%
|x|}{\sqrt{b}}\right) ^{-s}ds}{s-1} \\
&=&\frac{dx}{\sqrt{b}}\sum_{\nu =-1}^{\infty }\lim_{s\rightarrow
-\nu
}(s+\nu )\frac{\Gamma (s)}{s-1}\left( \frac{|x|}{\sqrt{b}}\right) ^{-s} \\
&=&\frac{dx}{\sqrt{b}}\sum_{\nu =-1}^{\infty }\lim_{s\rightarrow -\nu }\frac{%
\Gamma (s+\nu +1)}{(s+\nu -1)....s-1}\left(
\frac{|x|}{\sqrt{b}}\right) ^{-s}
\\
&=&\frac{dx}{|x|}\sum_{\nu =-1}^{\infty }\frac{(-1)^{\nu +1}}{(\nu +1)!}%
\left( \frac{|x|}{\sqrt{b}}\right) ^{\nu
+1}=\frac{dx}{|x|}e^{-|x|/\sqrt{b}}.
\end{eqnarray*}%
Again, for $\theta =-\alpha $ and by letting $\alpha \rightarrow
1^{+},$ we can easily obtain from (\ref{hh2}) formula (\ref{mu}),
by considering also (1.60) in \cite{MAT}.
\end{remark}

\subsection{Inhomogeneous Variance Gamma process}
For $\theta =0$ and $\alpha =2$, we have the important special case
represented by the time-inhomogeneous VG process (hereafter $VG^{I}$),
which we define as $\left\{ B(\mathit{%
\Gamma }^{I}(t)),t\geq 0\right\} ,$ where $B$ is a Brownian
motion such that $B(t)\sim \mathcal{N}(0,2t)$. We note that a pioneering definition of this process
can be found in  \cite{MAD2}. Its characteristic function reads%
\begin{equation}
\mathbb{E}e^{i\xi B(\mathit{\Gamma }^{I}(t))}=\exp \left\{ -\int_{0}^{t}\ln
\left( 1+b(s)\xi ^{2}\right) ds\right\}  \label{vg}
\end{equation}%
and its transition operator satisfies the following initial value problem%
\begin{equation*}
\left\{
\begin{array}{l}
\frac{\partial }{\partial t}u(x,t)=-\ln \left( 1-b(t)\frac{\partial ^{2}}{%
\partial x^{2}}\right) u(x,t) \\
u(x,s)=f(x).%
\end{array}%
\right.
\end{equation*}%
Thus the generator of $\left\{ B(\mathit{\Gamma ^I}(t)),t\geq 0\right\} $ can be written as $\mathcal{A}_{t}=-\ln \left(
1-b(t)\frac{\partial ^{2}}{\partial x^{2}}\right) ,$ by recalling the
representation (\ref{shilling}) together with (\ref{derivative}).

The mean-square displacement of $VG^{I}$ can be evaluated as follows:%
\begin{equation*}
\mathbb{E}B^{2}(\mathit{\Gamma }^{I}(t))=\mathbb{E}\mathit{\Gamma }%
^{I}(t)=\int_{0}^{t}b(s)ds,
\end{equation*}%
which is finite, by assumption. We note that the choice of $b(t)$
determines the asymptotic properties of the process,
which can be either diffusive, sub-diffusive, or super-diffusive.

We further observe that, also in the non-homogeneous
case, the $VG^I$ process can be represented as the difference of two independent
inhomogeneous gamma subordinators $\mathit{\Gamma }_{1}^{I}$ and $\mathit{\Gamma }%
_{2}^{I}, $  each having the following
characteristic functions
\begin{align*}
 \mathbb{E}e^{i\xi \mathit{\Gamma }_1^{I}(t)}=\mathbb{E}e^{i\xi \mathit{\Gamma }_2^{I}(t)}
 =\exp \left\{ -\int_{0}^{t}\ln \left( 1- i\sqrt{b(s)}\xi \right)
ds\right\}
\end{align*}
 Indeed, we can write%
\begin{eqnarray*}
\mathbb{E}e^{i\xi \mathit{\Gamma }_{1}^{I}(t)-i\xi \mathit{\Gamma }%
_{2}^{I}(t)} &=&\exp \left\{ -\int_{0}^{t}\left[ \ln \left( 1-i\sqrt{b(s)}%
\xi \right) +\ln \left( 1+i\sqrt{b(s)}\xi \right) \right] ds\right\} \\
&=&\exp \left\{ -\int_{0}^{t}\ln \left[ \left( 1-i\sqrt{b(s)}\xi \right)
\left( 1+i\sqrt{b(s)}\xi \right) \right] ds\right\} \\
&=&\exp \left\{ -\int_{0}^{t}\ln \left( 1+b(s)\xi ^{2}\right) ds\right\} =%
\mathbb{E}e^{i\xi B(\mathit{\Gamma }^{I}(t))}
\end{eqnarray*}%
and this property is very important for financial applications, in order to
model stochastic volatility.

Under the additional assumption that $b(t)<K,$ for any $t\geq 0$ and for a
constant $K<1,$ it is easy to check that the moment generating function of
$VG^{I}$ is finite for any $|\gamma |\leq \frac{1}{\sqrt{k}}$:
\begin{equation*}
\mathbb{E}e^{\gamma B(\mathit{\Gamma }^{I}(t))}=\exp \left\{ \int_{0}^{t}\ln
\left( 1-\gamma ^{2}b(s)\right) ds\right\} <\infty .
\end{equation*}

As a consequence of the subordination by the gamma process, the VG process has infinitely many small jumps
and a finite number of large jumps. The subordination implies the introduction of a new parameter (with respect to the Brownian
case) and enables the VG model to capture the negative skewness
and excess kurtosis, which are often displayed, in financial applications, by the log returns. The
variance parameter of the gamma process controls the degree of randomness
of subordination: indeed large values of the variance
result in fatter tails of the density. This feature is confirmed, in the
inhomogeneous case, by considering Theorem 13, for $\alpha =2$ and $\theta =0$.

\end{document}